\documentclass[11pt,oneside]{article}

\usepackage[utf8]{inputenc} \usepackage[T1]{fontenc}    \usepackage[english]{babel}

\usepackage{amsfonts}
\usepackage{nicefrac}       \usepackage{microtype}      \usepackage{lmodern}
\usepackage{amssymb,amsmath,amsthm}

\usepackage{stmaryrd}
\SetSymbolFont{stmry}{bold}{U}{stmry}{m}{n}

\usepackage{bbm,bm}
\usepackage{latexsym}
\usepackage{xcolor}

\usepackage{enumerate}
\usepackage{verbatim}
\usepackage{booktabs}       

\usepackage{url}            \usepackage[colorlinks=true]{hyperref}
\usepackage[numbers,sort&compress,square,comma]{natbib}
\usepackage{cleveref}

\usepackage{parskip}
\usepackage{geometry}
\usepackage[short]{optidef}
\usepackage{algorithmic,algorithm}

\usepackage{thmtools}

\declaretheorem{theorem}

\declaretheorem{lemma}
\declaretheorem{conjecture}

\declaretheoremstyle[qed=$\square$]{definitionwithend}

\crefname{assumption}{Assumption}{Assumptions}
\crefname{setting}{Setting}{Setting}
\crefname{conjecture}{Conjecture}{Conjectures}
\crefname{fact}{Fact}{Facts}

\newcommand{\abs}[1]{\ensuremath{\left\lvert #1 \right\rvert}}

\newcommand{\by}{\times}
 
\newcommand{\norm}[1]{\ensuremath{\left\lVert #1 \right\rVert}}
\newcommand{\ip}[1]{\ensuremath{\left\langle #1 \right\rangle}}

\newcommand{\set}[1]{\left\{#1\right\}}

\def\R{{\mathbb{R}}}

\newcommand{\framedheader}[3]{
  \framebox[\textwidth]{
    \vbox{
      \vspace{2mm}
      \hbox to \textwidth {\hspace{1em}\today \hfill #1\hspace{1em}}
      \vspace{4mm}
      \hbox to \textwidth {\hfill \Large{#2} \hfill}
      \vspace{2mm}
    }
  }
  \vspace*{4mm}
} \usepackage{letltxmacro}
\LetLtxMacro\orgvdots\vdots
\LetLtxMacro\orgddots\ddots

\makeatletter
\DeclareRobustCommand\vdots{\mathpalette\@vdots{}}
\newcommand*{\@vdots}[2]{\sbox0{$#1\cdotp\cdotp\cdotp\m@th$}\sbox2{$#1.\m@th$}\vbox{\dimen@=\wd0 \advance\dimen@ -3\ht2 \kern.5\dimen@
\dimen@=\wd2 \advance\dimen@ -\ht2 \dimen2=\wd0 \advance\dimen2 -\dimen@
    \vbox to \dimen2{\offinterlineskip
      \copy2 \vfill\copy2 \vfill\copy2 }}}
\DeclareRobustCommand\ddots{\mathinner{\mathpalette\@ddots{}\mkern\thinmuskip
  }}
\newcommand*{\@ddots}[2]{\sbox0{$#1\cdotp\cdotp\cdotp\m@th$}\sbox2{$#1.\m@th$}\vbox{\dimen@=\wd0 \advance\dimen@ -3\ht2 \kern.5\dimen@
\dimen@=\wd2 \advance\dimen@ -\ht2 \dimen2=\wd0 \advance\dimen2 -\dimen@
    \vbox to \dimen2{\offinterlineskip
      \hbox{$#1\mathpunct{.}\m@th$}\vfill
      \hbox{$#1\mathpunct{\kern\wd2}\mathpunct{.}\m@th$}\vfill
      \hbox{$#1\mathpunct{\kern\wd2}\mathpunct{\kern\wd2}\mathpunct{.}\m@th$}}}}

\DeclareRobustCommand\bddots{\mathinner{\mathpalette\@bddots{}\mkern\thinmuskip
  }}
\newcommand*{\@bddots}[2]{\sbox0{$#1\cdotp\cdotp\cdotp\m@th$}\sbox2{$#1.\m@th$}\vbox{\dimen@=\wd0 \advance\dimen@ -3\ht2 \kern.5\dimen@
\dimen@=\wd2 \advance\dimen@ -\ht2 \dimen2=\wd0 \advance\dimen2 -\dimen@
    \vbox to \dimen2{\offinterlineskip
      \hbox{$#1\mathpunct{\kern\wd2}\mathpunct{\kern\wd2}\mathpunct{.}\m@th$}\vfill
      \hbox{$#1\mathpunct{\kern\wd2}\mathpunct{.}\m@th$}\vfill
      \hbox{$#1\mathpunct{.}\m@th$}}}}
\makeatother \usepackage{soul}

\usepackage{subcaption}

\newcommand{\Nvalue}{20160}
\newcommand{\bbracket}[1]{\ensuremath{\left\llbracket #1 \right\rrbracket}}

\title{A Strengthened Conjecture on the Minimax Optimal  Constant Stepsize for Gradient Descent}
\date{\today}
\author{
	Benjamin Grimmer\footnote{Johns Hopkins University, Department of Applied Mathematics and Statistics, \texttt{grimmer@jhu.edu}}
	\and
	Kevin Shu\footnote{Georgia Institute of Technology, School of Mathematics, \texttt{kshu8@gatech.edu}}
	\and
	Alex L.\ Wang\footnote{Purdue University, Daniels School of Business, \texttt{wang5984@purdue.edu}}
}

\begin{document}
	
	\maketitle
	
	\begin{abstract}
		\citet{drori2012PerformanceOF} conjectured that the minimax optimal constant stepsize for $N$ steps of gradient descent is given by the stepsize that balances performance on Huber and quadratic objective functions. This was numerically supported by semidefinite program (SDP) solves of the associated performance estimation problems up to $N\approx 100$. This note presents a 
strengthened version of the initial conjecture. Specifically, we conjecture the existence of a certificate for the convergence rate with a very specific low-rank structure.
This structure allows us to bypass SDPs and to numerically verify both conjectures up to
$N=\Nvalue$. 	\end{abstract}

	\section{Introduction}
Gradient descent has recently seen renewed interest. Let $f\colon\mathbb{R}^n\rightarrow \mathbb{R}$ be an $L$-smooth convex function. Given an initialization $x_0$ and stepsizes $h_0,h_1,\dots h_{N-1}$, gradient descent defines
\begin{equation*}
	x_{k+1} = x_k - \frac{h_k}{L}\nabla f(x_k) \qquad \forall k=0,\dots, N-1. 
\end{equation*}
The behavior of gradient descent depends crucially on the choice of stepsizes~\cite{drori2012PerformanceOF,Teboulle2022,altschuler2023accelerationPartII,Grimmer2023-long,Grimmer2024-bestAccel}.

This note addresses the open problem of determining the minimax optimal constant stepsize for gradient descent: Given $N$, identify the stepsize $\alpha(N)\in\R$ minimizing the worst-case final objective gap 
\begin{equation} \label{eq:minimax-design}
	\min_{\tilde \alpha\in\mathbb{R}} \max_{(f,x_0)\in\mathcal{F}_{L,D}} f(x_N) - \inf f,
\end{equation}
where $\mathcal{F}_{L,D}$ denotes the set of $(f,x_0)$ such that $f$ is an $L$-smooth convex function with a minimizer $x_\star$ satisfying $\|x_0-x_\star\|\leq D$ and $x_N$ is the output of $N$ steps of gradient descent with constant stepsize $h_k=\tilde \alpha$.

The optimal stepsize $\alpha(N)$ was conjectured by~\cite{drori2012PerformanceOF} to be the one that achieves the same final error on the quadratic function and Huber function (parameterized by $\delta>0$)
\[Q(x)=\frac{L}{2}x^2 \qquad \text{and} \qquad H_\delta(x) = \begin{cases} L\delta |x|-\frac{L\delta^2}{2} &\text{ if }|x| \geq \delta\\ \frac{L}{2}x^2 &\text{ if }|x| \le \delta.\end{cases}
\] 
Note that $x_\star = 0$ is the unique minimizer of both $Q$ and $H_\delta$. Thus, setting $x_0=D$, we have that $(Q,x_0),(H_{\delta},x_0)\in\mathcal{F}_{L,D}$. When $f=Q$ and $x_0=D$, gradient descent with $h_k = \tilde\alpha$ has
\begin{align}
    \label{eq:quadratic_performance}
f(x_N) - f(x_\star) = \frac{(1-\tilde\alpha)^{2N}}{2} LD^2.
\end{align}
When $f=H_{\delta}$ with $\delta=\frac{D}{2N\tilde\alpha+1}$ and $x_0=D$, gradient descent with $h_k = \tilde\alpha$ has
\begin{align}
\label{eq:huber_performance}
f(x_N) - f(x_\star) = \frac{1}{2(2N\tilde\alpha+1)}LD^2.
\end{align}
Setting these two objective gaps equal gives the conjectured optimal stepsize $\alpha(N)$.
\begin{conjecture} [{\cite[Conjecture 3.1]{drori2012PerformanceOF}}]\label{conj:main}
	For any $N,L,D$, let $\alpha(N)\geq1$ denote the unique solution of $\frac{1}{2(2N\alpha+1)} = \frac{1}{2}(1-\alpha)^{2N}$ and let $r(N)$ denote their common value. Then, $\alpha(N)$ is the unique minimizer of~\eqref{eq:minimax-design} and achieves the value
	\begin{equation*}
		\min_{\tilde \alpha\in\mathbb{R}} \max_{(f,x_0)\in\mathcal{F}_{L,D}} f(x_N) - \inf f = r(N) LD^2.
	\end{equation*}
\end{conjecture}
A similar conjecture on the optimal stepsize for minimizing final gradient norm $\|\nabla f(x_N)\|$  was made by~\cite{Taylor2015SmoothSC}. A variant of this was recently proved by Rotaru et al.~\cite[Theorem 2.1]{rotaru2024exactworstcaseconvergencerates}.

 	\section{Preliminaries}
By a rescaling, we will take $L=1$ and $D=1$ in the remainder of this note without loss of generality. We fix $N$ and write $\alpha$ and $r$ instead of $\alpha(N)$ and $r(N)$.

Noting that \eqref{eq:quadratic_performance} is increasing on $\tilde \alpha\in[1,\infty)$ and \eqref{eq:huber_performance} is decreasing on $\tilde\alpha\in[0,\infty)$, one can deduce that no $\tilde \alpha$ can outperform the conjectured rate $r$ of $\alpha$. Formally:
\begin{lemma}\label{lem:lowerBound}
    For any $\tilde\alpha\in\R$,
	gradient descent with constant stepsize $h_k=\tilde\alpha$ has
	\[
	\max_{(f,x_0)\in\mathcal{F}_{1,1}} f(x_N) - \inf f \geq r.
	\]
\end{lemma}

Thus, \cref{conj:main} requires us to prove that gradient descent with constant stepsize $\alpha$ has
\begin{align}
    \label{eq:rate}
    \max_{(f,x_0)\in\mathcal{F}_{1,1}} f(x_N) - \inf f \leq r.
\end{align}
\paragraph{Performance estimation}
The performance estimation program (PEP) framework, pioneered in~\cite{taylor2017interpolation,drori2012PerformanceOF}, gives us a tool for proving statements of this form.
Fix an arbitrary $1$-smooth convex function $f$, initial iterate $x_0$, and minimizer $x_\star$ of $f$.
Let $x_0, \dots,x_N$ denote the iterates of gradient descent on $f$ with constant stepsize $\alpha$.
Let $f_i=f(x_i)$ and $g_i=\nabla f(x_i)$ for any $i\in \{\star,0,1,\dots N\}$.
Then, the following quantity must be nonnegative for all $i,j\in \{\star,0,1,\dots N\}$
\[
Q_{ij} = f_i - f_j - \langle g_j, x_i - x_j\rangle - \frac{1}{2}\|g_i - g_j\|^2.
\]

We will attempt to certify $f_N-f_\star \leq r$ by finding $\lambda_{ij} \ge 0$ so that 
\[
\sum_{ij} \lambda_{ij} Q_{ij} = f_\star-f_N + r \|x_0 - x_\star\|^2- Z,
\]
where $Z$ is a positive semidefinite quadratic form in the input data.
We refer to the matrix $\lambda\in\R^{(N+2)\by (N+2)}$, whose columns and rows are each indexed by $\set{\star,0,1,2,\dots,N}$, as a \emph{certificate}.
Note that such a certificate is a formal proof of \eqref{eq:rate}.

Constructing these certificates is highly nontrivial and involves guessing analytic expressions for numerical solutions to large semidefinite programs (SDPs).
This approach has been carried out successfully for various algorithms by several prior works including~\cite{altschuler2023accelerationPartI,Kim2016optimal,Lieder2020OnTC,Taylor2019StochasticFM,Dragomir2019OptimalCA,Gu2020tightPPM}.

\section{A Strengthened Conjecture}
Proving \cref{conj:main} directly via PEP requires finding the two matrices $\lambda$ and $Z$. This requires identifying (numerically or analytically) the $O(N^2)$ entries of these matrices.
A core challenge is that these certificates are not unique so that patterns in these certificates may not immediately reveal themselves in SDP solves.
Our strengthened conjecture below identifies a highly-structured low-rank $\lambda$ and $Z$, specified by only $O(N)$ entries.

\begin{conjecture}\label{conj:strengthened}
There exists a collection of multipliers $\lambda_{i,j}\geq 0$ of the form
\begin{align*}
    \lambda = \begin{pmatrix}
    0 & c_0 & c_1 & c_2 & c_3 & \dots & c_N\\
    0 & 0 & a_0 & d_0 c_2 & d_0 c_3 & \dots & d_0 c_N\\
    & b_0 & 0 & a_1 & d_1 c_3 & \dots & d_1c_N\\
    & & b_1 & \ddots & \ddots & \ddots & \vdots\\
    & &  & \ddots & \ddots & \ddots & d_{N-2}c_N\\
    & & & & b_{N-2}& 0& a_{N-1}\\
    & & & & & 0 & 0
    \end{pmatrix},
\end{align*}
such that
\begin{align}
    \label{eq:target_identity}
    \sum_{ij} \lambda_{ij} Q_{ij} = f_\star - f_N + r \left(\|x_0 - x_\star\|^2 - \norm{\left(x_0-\frac{1}{2r}\sum_{i=0}^Nc_ig_i\right)-x_\star}^2\right).
\end{align}
Here, $a, b,c,d$ are positive vectors of the appropriate dimensions.
\end{conjecture}

Note that the slack term $Z$ is a rank-one matrix according to \cref{conj:strengthened}.

The following theorem eliminates the quantities $a,b,c$ in \cref{conj:strengthened} to get an equivalent conjecture in only the $d$ vector.
Its proof expands the quantities $Q_{i,j}$ and equates the terms on the left and right of \eqref{eq:target_identity}.
\begin{theorem}\label{thm:main}
Let $a,b,c,d$ denote positive vectors of the appropriate dimensions.
Additionally, let $\epsilon_0,\dots,\epsilon_N\in\R$.
Then,
\begin{align}
    \label{eq:target_with_errors}
	\sum_{ij} \lambda_{ij}Q_{ij} &= f_\star - f_N + r\left(\norm{x_0 - x_\star}^2 - \norm{\left(x_0 - \frac{1}{2r}\sum_{i=0}^Nc_i g_i\right) - x_\star}^2 \right)\\
	&\qquad + \sum_{i=0}^{N-1} \epsilon_i \left(f_i - f_\star\right) + \frac{\epsilon_N}{2}\norm{g_0}^2\nonumber
\end{align}
if and only if
\begin{gather} \tag{A}
    \begin{cases}
        a_i = \frac{1}{\alpha}\Big(\frac{1}{2r}c_{i+1}^2 + \frac{1}{2r}c_ic_{i+1} - a_{i+1}  - (1+\alpha)c_{i+1}\left(1 + \sum_{j=0}^{i-1}d_j\right) \\
        \qquad\qquad - d_{i+1} \left(\sum_{j=i+3}^{n}c_j\right) + b_{i+1} (2\alpha - 1)\Big) \qquad\forall i\in[0,N-3],\\
        a_{N-2}= \frac{1}{\alpha}\left(\frac{1}{2r}c_{N-1}^2  + \frac{1}{2r}c_{N-2}c_{N-1} - a_{N-1} - (1+\alpha)c_{N-1}\left(1+\sum_{j=0}^{N-3}d_j\right)\right),\\
        a_{N-1} = 1 - c_N\left(1 + \sum_{j=0}^{N-2}d_j\right)
    \end{cases}
\end{gather}
\begin{gather} \tag{B}
    \begin{cases}
        b_i= \frac{1}{\alpha}\Big(
            \frac{\alpha-1}{2r}c_{i+1}^2  - \frac{1}{2r}c_ic_{i+1} - (\alpha - 1)a_{i+1} +c_{i+1}\left(1+\sum_{j=0}^{i-1}d_j\right)\\
            \qquad\qquad - (\alpha - 1)d_{i+1} \left(\sum_{j=i+3}^{n}c_j\right) + (\alpha - 1)b_{i+1} (2\alpha - 1)\Big)\qquad\forall i\in[0,N-3]\\
            b_{N-2}= \frac{1}{\alpha}\left(
            \frac{\alpha - 1}{2r}c_{N-1}^2  - \frac{1}{2r}c_{N-2}c_{N-1} - (\alpha-1)a_{N-1} + c_{N-1}\left(1+\sum_{j=0}^{N-3}d_j\right)\right)
    \end{cases}
\end{gather} 
\begin{gather} \tag{C}
    \begin{cases}
        c_i = 2r\left(\alpha \sum_{\ell=0}^{i}d_\ell  - d_i + \alpha\right)\qquad\forall i \in[0,N-2]\\
    c_{N-1}= 2r\left(1 + \sum_{j=0}^{N-2}d_j + \frac{\alpha -1}{\sqrt{2r}}\right)\\
    c_N = \sqrt{2r}
    \end{cases}
\end{gather}
and
\begin{gather} \tag{D}
    \begin{cases}
        \epsilon_0 = a_0 + d_0 \sum_{j=2}^N c_j- b_0 - c_0 \\
        \epsilon_i = b_{i-1} + a_{i}+d_{i}\sum_{j=i+2}^N c_j - a_{i-1} - b_{i} - c_{i}\left(1 + \sum_{j=0}^{i-2}d_j\right) \qquad \forall i\in[1,\dots, N-2]\\
        \epsilon_{N-1}=b_{N-2} + a_{N-1} - a_{N-2} - c_{N-1}\left(1+\sum_{j=0}^{N-3}d_j\right)\\
        \epsilon_N=- c_0 - a_0 - d_0 \sum_{j=2}^N c_j + (2\alpha-1)b_0 +\frac{1}{2r}c_0^2.
    \end{cases}
\end{gather}
\end{theorem}
\begin{proof}
Let $\Delta$ denote the difference between the LHS and RHS of \eqref{eq:target_with_errors}. Let $\bbracket{\Delta}_f$ denote the terms in $\Delta$ involving $f_i$ for $i\in\set{\star,0,\dots,N}$ and let $\bbracket{\Delta}_g$ denote the remainder of the terms.

$\bbracket{\Delta}_f$ evaluates to zero if and only if (i) for $i \in[0,N-1]$ the difference between the $i$th row sum and the $i$th column sum is $\epsilon_i$ and (ii) the $N$th column of $\lambda$ has unit sum.
Note that (i) holds if and only if the first three equations of system D holds.
Note that (ii) holds if and only if the last equation of system A holds.

Next, we turn to $\bbracket{\Delta}_g$. 
Expanding the definition of $\lambda_{ij}$ and $Q_{ij}$ gives
\begin{align*}
    \bbracket{\Delta}_g
    =&\sum_{i=0}^N c_i \left(\ip{g_i, x_0 - x_\star - \alpha\sum_{\ell = 0}^{i-1} g_\ell} - \frac{1}{2}\norm{g_i}^2\right)\\
    &\qquad+\sum_{i=0}^{N-1} a_{i} \left(-\alpha\ip{g_{i+1}, g_i} - \frac{1}{2}\norm{g_i - g_{i+1}}^2\right)\\
    &\qquad+\sum_{i=0}^{N-2}\sum_{j=i+2}^N d_i c_j\left(-\sum_{\ell=i}^{j-1}\alpha\ip{g_j,g_\ell} - \frac{1}{2}\norm{g_i-g_j}^2\right)\\
    &\qquad+\sum_{i=0}^{N-2} b_{i} \left(\alpha\norm{g_{i}}^2 - \frac{1}{2}\norm{g_{i+1} - g_{i}}^2\right) - \ip{x_0 - x_\star, \sum_{i=0}^N c_i g_i} + \frac{1}{4r}\norm{\sum_{i=0}^N c_ig_i}^2 - \frac{\epsilon_N}{2}\norm{g_0}^2.
\end{align*}
The term involving $x_0 - x_\star$ is zero. Thus, $\bbracket{\Delta}_g$ is a quadratic form in $g_0,\dots,g_N$ only.

Let $i<j$ and $\abs{i - j}\geq 2$, then the coefficient on $\ip{g_i,g_j}$ is
\begin{align*}
   c_j\left(-\alpha  + \frac{1}{2r}c_i + d_i - \alpha \sum_{\ell=0}^{i}d_\ell \right).
\end{align*}
These coefficients are zero if and only if the first equation in system C holds.

The coefficient on $\frac{1}{2}\norm{g_N}^2$ is
\begin{align*}
    -c_N - a_{N-1} - c_N \sum_{i=0}^{N-2} d_i +\frac{1}{2r} c_N^2.
\end{align*}
Substituting $a_{N-1} = 1 - c_N(1 + \sum_{j=0}^{N-2}d_j)$, we have that this coefficient is
\begin{align*}
    -1 + \frac{1}{2r}c_N^2.
\end{align*}
This coefficient is zero if and only if the last equation in system C holds.

The coefficient on $\ip{g_{N-1}, g_N}$ is
\begin{align*}
    - (\alpha-1) a_{N-1}-\alpha c_N \left(1 + \sum_{i=0}^{N-2}d_i\right) + \frac{1}{2r} c_{N-1}c_N.
\end{align*}
Substituting $c_N = \sqrt{2r}$ and $a_{N-1} = 1 - c_N(1 + \sum_{j=0}^{N-2}d_j)$, we have that this coefficient is
\begin{align*}
    - (\alpha-1) - \sqrt{2r}\left(1 + \sum_{j=0}^{N-2}d_j\right) + \frac{1}{\sqrt{2r}} c_{N-1}.
\end{align*}
This is equal to zero if only if the second equation in system C holds.

The coefficients on $\frac{1}{2}\norm{g_{N-1}}^2$ and $\ip{g_{N-2}, g_{N-1}}$ are
\begin{gather*}
    - a_{N-2} - b_{N-2} 
    -a_{N-1} - c_{N-1}\left(1+\sum_{i=0}^{N-3}d_i\right)+ \frac{1}{2r}c_{N-1}^2\qquad \text{and}\\
     -(\alpha-1) a_{N-2} + b_{N-2}  -\alpha c_{N-1} \left(1+\sum_{i=0}^{N-3}d_i\right)+ \frac{1}{2r}c_{N-2}c_{N-1}
\end{gather*}
respectively.
As $\alpha\neq 0$, we can invert this as a linear system of two equations in the two variables $a_{N-2}$ and $b_{N-2}$. Thus, these two coefficients are equal to zero if and only if the second equation in system A and the second equation in system B hold.

Now, let $i\in[0,N-3]$. The coefficients on $\frac{1}{2}\norm{g_{i+1}}^2$ and $\ip{g_{i},g_{i+1}}$ are
\begin{gather*}
    - a_i- b_i - a_{i+1} - \sum_{j=i+3}^N d_{i+1}c_j - c_{i+1}\left(1+\sum_{j=0}^{i-1}d_j\right) + (2\alpha-1) b_{i+1}  +\frac{1}{2r}c_{i+1}^2\qquad\text{and}\\
    - (\alpha-1)a_i+b_i - \alpha c_{i+1}\left(1+\sum_{j=0}^{i-1}d_j\right)+\frac{1}{2r}c_ic_{i+1}
\end{gather*}
respectively. As $\alpha\neq 0$, we can invert this as a linear system of two equations in the two variables $a_{i}$ and $b_{i}$. Thus, these two coefficients are equal to zero if and only if the first equation in system A and the first equation in system B hold.

All that remains is the coefficient on $\frac{1}{2}\norm{g_0}^2$. This coefficient is given by
\begin{align*}
    - c_0 - a_0 - d_0 \sum_{j=2}^N c_j + (2\alpha-1)b_0 +\frac{1}{2r}c_0^2 - \epsilon_N.
\end{align*}
Setting this coefficient equal to zero gives the last equation of system D.
\end{proof} 	\Cref{thm:main} eliminates the variables $a,b,c$ in terms of $d$: Suppose a positive vector $d$ is given.
We may use $d$ to \emph{define} $a,b,c$ as follows:
Use system C to define $c$.
Next, use system A equation 3 to define $a_{N-1}$
and use system A equation 2 and system B equation 2 to define $a_{N-2}$ and $b_{N-2}$.
To complete the definition of $a,b,c$ in terms of $d$, we use system A equation 1 and system B equation 1 starting from $i= N-2$ down to $i=0$ to define the remaining $a_i$ and $b_i$.
Finally, 
define the errors $\epsilon_0,\dots,\epsilon_N$ by system D.
One may check that the entries of $c$ are affine in $d$ and the entries of $a,b$ are quadratic in $d$. Thus, the errors $\epsilon_0,\dots,\epsilon_N$ are quadratic expressions in $d$.

We refer to $d$ as a certificate if each of the errors $\epsilon_0=\dots=\epsilon_N=0$ and the vectors $a,b,c,d$ are positive.
Relaxing this to allow for numerical imprecision, we say $d$ is a $\delta$-certificate if $\sum_{i=0}^N \max\{\epsilon_i,0\} \leq \delta$ and $a,b,c$ are positive.

A $\delta$-certificate formally proves a relaxation of the rate claimed in Conjecture~\ref{conj:main}:
Since gradient descent is a descent method for $\alpha\in(0,2)$, we have the bounds $0\leq f_i - f_\star \leq f_0 - f_\star \leq \frac{\norm{x_0 - x_\star}^2}{2} \leq \frac{1}{2}$ and $0\leq \frac{\norm{g_0}^2}{2} \leq \frac{\norm{x_0 - x_\star}^2}{2} \leq \frac{1}{2}$. Hence, a $\delta$-certificate $d$ implies
\begin{align*}
	f_N - f_\star \leq r + \frac{1}{2}\sum_{i=0}^N \max\{\epsilon_i,0\} = r + \frac{\delta}{2}.
\end{align*}

\section{Numerical Construction of $\delta$-Certificates}\label{sec:numerics}
For each $N$, we can apply Newton's method to the system of quadratic equalities defined by system D after eliminating $a,b,c$ via systems A,B,C and setting $\epsilon_0,\dots,\epsilon_N = 0$. Despite this system being overdetermined ($N+1$ equations in $N-1$ variables), Newton's method was able to find solutions up to machine precision.

The general shape of the certificates change only mildly as $N$ changes; see Figure~\ref{fig:four_plots}.
Thus, we are able to generate a good initialization for Newton's method by linearly extrapolating certificates for smaller values of $N$ to larger values of $N$ (implementation available at \url{www.github.com/bgrimmer/GDConstantStepsizeCertificate}). We do this for every value $N=3,4,5,\dots,2240$; every 320th value $N=2240, 2560 ,\dots, 8960$; and every 1600th value $N=8960,10560,\dots ,\Nvalue$. Every computed $\delta$-certificate has individual errors $|\epsilon_i| \leq 10^{-15}$ and total positive error $\delta\leq 2\times 10^{-11}$, effectively only being limited by machine precision.
\begin{theorem}\label{thm:computed}
	For every $N$ listed above, up to $20160$, the conjectured optimal stepsize $\alpha(N)$ nearly attains the conjectured optimal rate $r(N)$:
	$$ \max_{(f,x_0)\in\mathcal{F}_{L,D}} f(x_N)-\inf f  \leq  \left(r(N)+10^{-11}\right)LD^2. $$
\end{theorem}
\begin{proof}
	See certificates at~\url{www.github.com/bgrimmer/GDConstantStepsizeCertificate}.
\end{proof}
\begin{figure}[t]
	\centering
	\begin{subfigure}[b]{0.24\textwidth}
		\includegraphics[width=\textwidth]{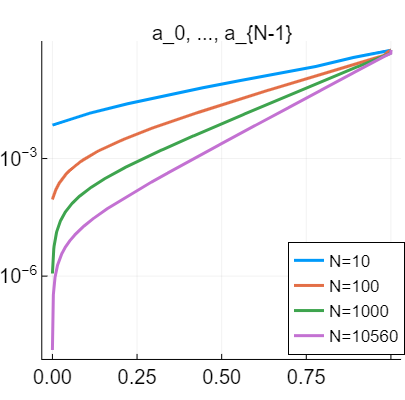}
		\caption{a's}
		\label{fig:a_val}
	\end{subfigure}
	\hfill
	\begin{subfigure}[b]{0.24\textwidth}
		\includegraphics[width=\textwidth]{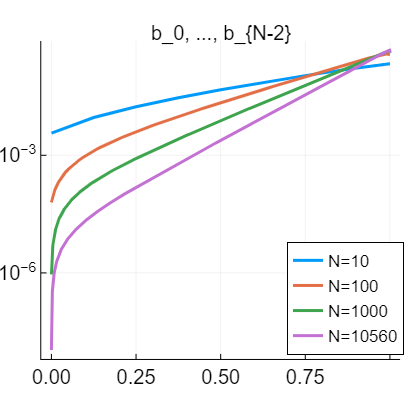}
		\caption{b's}
		\label{fig:b_val}
	\end{subfigure}
	\hfill
	\begin{subfigure}[b]{0.24\textwidth}
		\includegraphics[width=\textwidth]{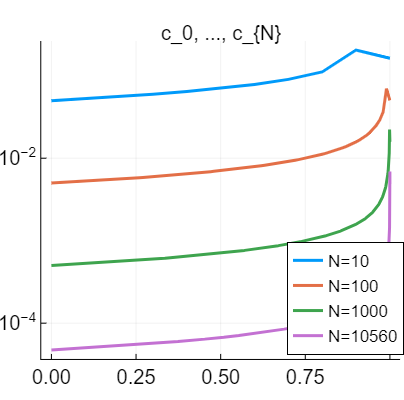}
		\caption{c's}
		\label{fig:c_val}
	\end{subfigure}
	\hfill
	\begin{subfigure}[b]{0.24\textwidth}
		\includegraphics[width=\textwidth]{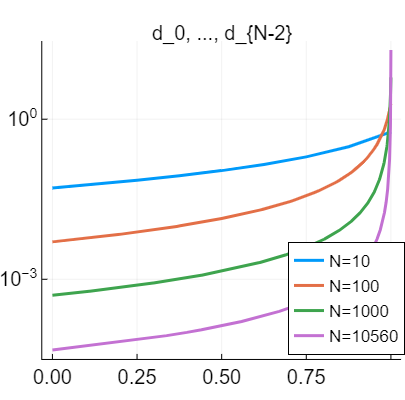}
		\caption{d's}
		\label{fig:d_val}
	\end{subfigure}
	\caption{Numerically computed certificates $a,b,c,d>0$ for various $N$, scaled to range $[0,1]$.}
	\label{fig:four_plots}
\end{figure}
 	
	\paragraph{Acknowledgements.} Benjamin Grimmer's work was supported in part by the Air Force Office of Scientific Research under award number FA9550-23-1-0531.
	
	\bibliographystyle{plainnat}

\begin{thebibliography}{14}
\providecommand{\natexlab}[1]{#1}
\providecommand{\url}[1]{\texttt{#1}}
\expandafter\ifx\csname urlstyle\endcsname\relax
  \providecommand{\doi}[1]{doi: #1}\else
  \providecommand{\doi}{doi: \begingroup \urlstyle{rm}\Url}\fi

\bibitem[Altschuler and
  Parrilo(2023{\natexlab{a}})]{altschuler2023accelerationPartI}
Jason~M. Altschuler and Pablo~A. Parrilo.
\newblock Acceleration by stepsize hedging i: Multi-step descent and the silver
  stepsize schedule, 2023{\natexlab{a}}.

\bibitem[Altschuler and
  Parrilo(2023{\natexlab{b}})]{altschuler2023accelerationPartII}
Jason~M. Altschuler and Pablo~A. Parrilo.
\newblock Acceleration by stepsize hedging ii: Silver stepsize schedule for
  smooth convex optimization, 2023{\natexlab{b}}.

\bibitem[Dragomir et~al.(2019)Dragomir, Taylor, d’Aspremont, and
  Bolte]{Dragomir2019OptimalCA}
Radu-Alexandru Dragomir, Adrien Taylor, Alexandre d’Aspremont, and
  J{\'e}r{\^o}me Bolte.
\newblock Optimal complexity and certification of bregman first-order methods.
\newblock \emph{Mathematical Programming}, 194:\penalty0 41 -- 83, 2019.

\bibitem[Drori and Teboulle(2012)]{drori2012PerformanceOF}
Yoel Drori and Marc Teboulle.
\newblock Performance of first-order methods for smooth convex minimization: a
  novel approach.
\newblock \emph{Mathematical Programming}, 145:\penalty0 451--482, 2012.

\bibitem[Grimmer(2024)]{Grimmer2023-long}
Benjamin Grimmer.
\newblock {Provably Faster Gradient Descent via Long Steps}.
\newblock \emph{To appear in SIAM Journal on Opimization}, 2024.

\bibitem[Grimmer et~al.(2024)Grimmer, Shu, and Wang]{Grimmer2024-bestAccel}
Benjamin Grimmer, Kevin Shu, and Alex~L. Wang.
\newblock {Accelerated Objective Gap and Gradient Norm Convergence for Gradient
  Descent via Long Steps}.
\newblock \emph{arxiv:2403.14045}, 2024.

\bibitem[Gu and Yang(2020)]{Gu2020tightPPM}
Guoyong Gu and Junfeng Yang.
\newblock Tight sublinear convergence rate of the proximal point algorithm for
  maximal monotone inclusion problems.
\newblock \emph{SIAM Journal on Optimization}, 30\penalty0 (3):\penalty0
  1905--1921, 2020.
\newblock \doi{10.1137/19M1299049}.
\newblock URL \url{https://doi.org/10.1137/19M1299049}.

\bibitem[Kim and Fessler(2016)]{Kim2016optimal}
Donghwan Kim and Jeffrey~A. Fessler.
\newblock Optimized first-order methods for smooth convex minimization.
\newblock \emph{Math. Program.}, 159\penalty0 (1–2):\penalty0 81–107, sep
  2016.
\newblock ISSN 0025-5610.
\newblock \doi{10.1007/s10107-015-0949-3}.
\newblock URL \url{https://doi.org/10.1007/s10107-015-0949-3}.

\bibitem[Lieder(2020)]{Lieder2020OnTC}
Felix Lieder.
\newblock On the convergence rate of the halpern-iteration.
\newblock \emph{Optimization Letters}, pages 1--14, 2020.

\bibitem[Rotaru et~al.(2024)Rotaru, Glineur, and
  Patrinos]{rotaru2024exactworstcaseconvergencerates}
Teodor Rotaru, François Glineur, and Panagiotis Patrinos.
\newblock Exact worst-case convergence rates of gradient descent: a complete
  analysis for all constant stepsizes over nonconvex and convex functions,
  2024.
\newblock URL \url{https://arxiv.org/abs/2406.17506}.

\bibitem[Taylor and Bach(2019)]{Taylor2019StochasticFM}
Adrien Taylor and Francis Bach.
\newblock Stochastic first-order methods: non-asymptotic and computer-aided
  analyses via potential functions.
\newblock \emph{ArXiv}, abs/1902.00947, 2019.

\bibitem[Taylor et~al.(2017)Taylor, Hendrickx, and
  Glineur]{taylor2017interpolation}
Adrien Taylor, Julien Hendrickx, and Fran\c{c}ois Glineur.
\newblock Smooth strongly convex interpolation and exact worst-case performance
  of first-order methods.
\newblock \emph{Mathematical Programming}, 161:\penalty0 307–345, 2017.

\bibitem[Taylor et~al.(2015)Taylor, Hendrickx, and Glineur]{Taylor2015SmoothSC}
Adrien~B. Taylor, Julien~M. Hendrickx, and François Glineur.
\newblock Smooth strongly convex interpolation and exact worst-case performance
  of first-order methods.
\newblock \emph{Mathematical Programming}, 161:\penalty0 307 -- 345, 2015.
\newblock URL \url{https://api.semanticscholar.org/CorpusID:9479850}.

\bibitem[Teboulle and Vaisbourd(2022)]{Teboulle2022}
Marc Teboulle and Yakov Vaisbourd.
\newblock An elementary approach to tight worst case complexity analysis of
  gradient based methods.
\newblock \emph{Math. Program.}, 201\penalty0 (1–2):\penalty0 63–96, oct
  2022.
\newblock ISSN 0025-5610.
\newblock \doi{10.1007/s10107-022-01899-0}.
\newblock URL \url{https://doi.org/10.1007/s10107-022-01899-0}.

\end{thebibliography}

\end{document}